\documentclass[psamsfonts]{amsart}

\usepackage{eucal}

\usepackage{amsmath}
\usepackage{amstext}
\usepackage{amssymb}
\usepackage{amsthm}
\usepackage{amscd}

\usepackage{longtable}
\RequirePackage[dvipsnames,usenames]{color}

\theoremstyle{plain}%default
\newtheorem{thm}{Theorem}[section]
\newtheorem{thm*}{Theorem}
\newtheorem{lem}[thm]{Lemma}

\newtheorem{prop*}[thm*]{Proposition}
\newtheorem{cor}[thm]{Corollary}

\theoremstyle{definition}
\newtheorem{defi}[thm]{Definition}

\newtheorem{qu}[thm]{Question}

\theoremstyle{remark}

\newtheorem{rmk}[thm]{Remark}

\DeclareMathOperator{\Hom}{Hom}

\DeclareMathOperator{\rank}{rank}

\DeclareMathOperator{\Tor}{Tor}
\DeclareMathOperator{\Ann}{ann}

\DeclareMathOperator{\HH}{H}
\DeclareMathOperator{\reg}{reg}
\DeclareMathOperator{\Sing}{Sing}
\DeclareMathOperator{\Supp}{Supp}
\newcommand{\Ipe}{I^{[p^e]}}

\newcommand{\twoVector}[2]
{\left[
\begin{array}{c} #1 \\ #2 \\
\end{array}
\right]}

\begin{document}

\title
[Castelnuovo-Mumford regularity and $F$-jumping coefficients]
{Castelnuovo-Mumford regularity and the discreteness of $F$-jumping coefficients in graded rings}

\author[M.~Katzman]{Mordechai Katzman}
\address[Katzman]{Department of Pure Mathematics,
University of Sheffield, Hicks Building, Sheffield S3 7RH, United Kingdom}
\email{M.Katzman@sheffield.ac.uk}

\author[W.~Zhang]{Wenliang Zhang}
\address[Zhang]{Department of Mathematics, University of Nebraska, Lincoln, NE 68588, USA}
\email{wzhang15@unl.edu}

\subjclass[2000]{Primary  13A35}

%\date{\today}

%\keywords{}

\thanks{
The results in this paper were obtained while both authors enjoyed the hospitality of the School of Mathematics
at the University of Minnesota.
The first author also wishes to acknowledge support through Royal Society grant TG102669. The second author is supported in part by NSF Grant DMS \#1068946.
}

\begin{abstract}
In this paper we show that the sets of $F$-jumping coefficients of ideals form discrete sets
in certain graded $F$-finite rings.
We do so by giving a criterion based on linear bounds for the growth of the Castelnuovo-Mumford regularity
of certain ideals. We further show that these linear bounds exists for one-dimensional rings and for  ideals of (most) two-dimensional domains.
We conclude by applying our technique to prove that all sets of $F$-jumping coefficients of all ideals in the determinantal ring  given as the quotient by $2\times 2$ minors in a $2\times 3$ matrix of indeterminates form discrete sets.
\end{abstract}

\maketitle
%\begin{abstract}
%\end{abstract}

%%%%%%%%%%%%%%%%%%%%%%%%%%%%%%%%%%%%%%%%%%%%%%%%%%%%%%%%%%%%%%%%%%%%%%%%%%%%%%%%%%%%%%%%%%%%%
\section{Introduction}\label{Section: Introduction}
%%%%%%%%%%%%%%%%%%%%%%%%%%%%%%%%%%%%%%%%%%%%%%%%%%%%%%%%%%%%%%%%%%%%%%%%%%%%%%%%%%%%%%%%%%%%%

The aim of this paper is to establish the discreteness of the set of $F$-jumping coefficients of ideals
in a certain class of graded $F$-finite rings which includes one-dimensional rings and two-dimensional domains.
The remainder of this introductory section will review the prerequisite notions necessary to understand the problem at hand and the methods used to solve them.

All rings in this paper are commutative and have prime characteristic $p$. If $S$ is such a ring and $M$ is an $S$-module, for all $e\geq 0$ we may define an $S$-bisubmodule $F^e_* M$, which is identical to $M$ as an Abelian group, and on which $S$ acts on the right with the
given action while the left action is given by $s \cdot a = s^{p^e} a$ for all $s\in S$ and $a\in M$.
We shall further assume in this paper that all rings $S$ are \emph{$F$-finite}, i.e.,
that $F_*^e S$ are finitely generated left $S$-modules for all $e\geq 0$.
Given a reduced ring $S$ as above we also define a (non-commutative) graded algebra
$\mathcal{C}_S= \oplus_{e\geq 0} \Hom_S(F_*^e S, S)$ where multiplication of $\phi\in \Hom_S(F_*^\alpha S, S)$ and
$\psi\in \Hom_S(F_*^\beta S, S)$ is defined as
$\phi \psi= \phi \circ F_*^\alpha(\psi)\in \Hom_S(F_*^{\alpha+\beta} S, S)$.
% where
%$F^\alpha(\psi)$ is the $S$-linear map $\Hom_S(F_*^{\alpha+\beta} S, F_*^{\alpha} S)$ isomorphic
%to $\psi$
(cf.~\cite[\S 3]{SchwedeTestIdealsInNonQGor} and \cite[\S 2]{BlickleTestIdealsViaAlgebras}.)
Note that $\mathcal{C}_S$ is an $S$-bimodule, as $S$ acts on each $\Hom_S(F_*^\alpha S, S)$
on the left via its left action on $F_*^\alpha S$ and on the right via its right action on $F_*^\alpha S$.

When $S$ is reduced we shall tacitly identify the inclusion $S \hookrightarrow F^e_* S$ given by
$s\mapsto s\cdot 1=s^{p^e}$
with the inclusion of rings $S \subset S^{1/p^e}$.

\begin{defi}\label{Definition: generalized test ideal}
Given a ring $S$ an ideal $\mathfrak{a}\subseteq S$ and a positive real number $t$, we define
the \emph{test ideal $\tau(\mathfrak{a}^t)$} to be the
unique smallest non-zero ideal $J \subseteq R$ such that
$\phi((\mathfrak{a}^{\lceil t (p^{e}-1) \rceil}J)^{1/p^e}) \subseteq J$ for all $e > 0$ and all
$\phi \in \Hom_S(F_*^e S, S)$.
\end{defi}

Our definition is inspired by the work of Karl Schwede (cf.~\cite{SchwedeTestIdealsInNonQGor}) and is equivalent to the
 original definition in \cite{HaraYoshidaGeneralizationOfTightClosure}.
It is not immediately clear why the unique smallest non-zero ideal $J \subseteq R$ in Definition
\ref{Definition: generalized test ideal} should exist; its existence is guaranteed by
\cite[Theorem 3.18]{SchwedeTestIdealsInNonQGor}.

These generalized test ideals and their characteristic-zero counterparts, multiplier ideals,
have recently attracted the attention of many algebraic geometers and algebraists. A specific direction of research
aims to relate the $F$-jumping coefficients of ideals defined below
to the geometrical properties of the varieties defined by these ideals.

Generalized test ideals satisfy the following basic properties:
\begin{thm}[cf.~Remark 2.12 in \cite{MustataTakagiWatanabeFThresholdsAndBernsteinSato} and Lemma 3.23 in \cite{DiscretenessAndRationalityOfFjumping}]
Let $\mathfrak{a}$ be any ideal.
\begin{enumerate}
\item[(a)] For all $0<s<t$,  $\tau(\mathfrak{a}^s)\supseteq \tau(\mathfrak{a}^t)$.
\item[(b)] For any $t>0$ there exists an $\epsilon>0$ such that
$\tau(\mathfrak{a}^c)=\tau(\mathfrak{a}^t)$ for all $t\leq c<t+\epsilon$.
\end{enumerate}
\end{thm}

These two properties above suggest the following:
\begin{defi}\label{Definition: F-jumping coefficient}
A positive real number $t$ is an $F$-jumping coefficient of the ideal $\mathfrak{a}$ if
$\tau(\mathfrak{a}^{t-\epsilon}) \supsetneqq \tau(\mathfrak{a}^t)$ for all $\epsilon>0$.
\end{defi}

The study of the nature of the set of $F$-jumping coefficients of a given ideal has recently attracted
intense attention. The question of the rationality and discreteness of these sets have recently been studied
in a number of papers, for example \cite{Hara2006}, \cite{BMS08}, \cite{BMS09}, \cite{KLZ09}, \cite{SchwedeDiscretenessQGor},
\cite{DiscretenessAndRationalityOfFjumping}, and  \cite{STZ11}.  In particular, the discreteness of $F$-jumping coefficients of an ideal in a non-$\mathbb{Q}$-Gorenstein ring remains open.

This paper links the question of the discreteness of sets of $F$-jumping coefficients
with the notion of Castelnuovo-Mumford regularity, a classical numerical invariant.
This notion has already played an important role in the study of rings of characteristic $p$. For example, in \cite{KatzmanComplexity} the connection between the linear growth of Castelnuovo-Mumford regularity of Frobenius powers of ideals and tight closure was explored. Specifically, the following question was raised there:
\begin{qu}[Question 2 in \cite{KatzmanComplexity}]
\label{KatzmanQuestion}
Let $R=k[x_1,\dots,x_n]$ be a polynomial ring over a field $k$ of characteristic $p>0$ and let $I,J$ be two homogeneous ideals of $R$. Does the Castelnuovo-Mumford regularity of $I+J^{[p^e]}$ grow linearly with $p^e$?
\end{qu}

Surprisingly, even a positive answer to the above question in the case when $I$ is principal will imply the discreteness of $F$-jumping coefficients, and this is one of our main results:
\newtheorem*{thm1}{\rm\bf Theorem \ref{Theorem: discreteness under linear regularity}}
\begin{thm1}
\label{maintheorem}
Let $R=\mathbb{K}[x_1,\dots,x_n]$ be a polynomial ring over an $F$-finite field $\mathbb{K}$ of prime characteristic $p$ and
$I=(g_1,\dots,g_{\nu})$ be a homogeneous ideal of $R$. Assume that
$\reg(R/(I^{[p^e]}+g_iR))\leq Cp^e$ for all $e\geq 1$ and $1\leq i\leq \nu$, where $C$ is a constant independent of $e$ and $i$. Then the sets of $F$-jumping coefficients of all ideals (homogeneous or not) in $S=R/I$ are discrete.
\end{thm1}

This theorem is proved by showing that its hypothesis imply the gauge-boundedness of $R/I$-- the discreteness of the sets of $F$-jumping coefficients is a corollary of this. We do not know whether the hypothesis of the Theorem are equivalent to the gauge-boundedness of $R/I$.

In general, Question \ref{KatzmanQuestion} is rather difficult, and not much progress has been made since it was first raised.
In this paper, using a deep result of Eisenbud-Huneke-Ulrich in \cite{EisenbudHunekeUlrichTheRegularityOfTor}, we give a positive answer to Question \ref{KatzmanQuestion} in a special case as follows:

\newtheorem*{cor1}{\rm\bf Corollary \ref{discreteness singular locus}}
\begin{cor1}
Let $R=\mathbb{K}[x_1,\dots,x_n]$ be a polynomial ring over an $F$-finite field $\mathbb{K}$ of prime characteristic $p$ and
$I=(g_1,\dots,g_{\nu})$ be a homogeneous ideal of $R$.  If $\dim(\Sing(R/(g_i))\cap V(I))\leq 1$ for all $1 \leq i \leq \nu$,
then  $\reg(R/(I^{[p^e]}+g_iR))\leq Cp^e$ for all $e\geq 1$ and $1\leq i\leq \nu$, where $C$ is a constant independent of $e$ and $i$. Consequently, the sets of $F$-jumping coefficients of all ideals (homogeneous or not) in $S=R/I$ are discrete.
\end{cor1}

This corollary immediately implies the discreteness of the set of jumping coefficients in graded one-dimensional $F$-finite rings and
(most) two-dimensional domains (Remark \ref{Remark: two-dimensional domains}.)

The last section in this paper shows the usefulness of our approach by applying Theorem \ref{Theorem: discreteness under linear regularity} to a determinantal ring.

\newtheorem*{cor2}{\rm\bf Corollary \ref{Corollary: determinantal ideal}}
\begin{cor2}
Let $R=\mathbb{K}[x_{ij}]$ with $1\leq i\leq 2$, $1\leq j\leq 3$. Let $I$ be the ideal generated by the $2\times 2$ minors of the matrix $(x_{ij})$. Then all sets of $F$-jumping coefficients of all ideals (homogeneous or not) in $R/I$ are discrete.
\end{cor2}

\section{Gauge boundedness and Castelnuovo-Mumford regularities}
\label{Section: Gauge boundedness and Mumford-Castelnuovo regularities}

In \cite{BlickleTestIdealsViaAlgebras}, Manuel Blickle introduced a notion called {\it gauge boundedness},  which has particular significance to
the work presented here, and which we now describe.
Henceforth, let $R$ denote the  polynomial ring $\mathbb{K}[x_1, \dots, x_n]$ over an $F$-finite field $\mathbb{K}$,
let $\mathfrak{m}$ denote the ideal $R$ generated by its variables,
and for all $d\geq 0$ let $R_d$ be the $\mathbb{K}$-vector subspace of $R$ spanned by all monomials
$x_1^{\alpha_1} \dots x_n^{\alpha_n}$ with
$0\leq \alpha_1, \dots , \alpha_n \leq d$.
Given an $R$-module $M$ generated by a finite set of elements $\{ m_1, \dots, m_k\}$, we define a filtration $\{M_d\}_{d\geq 0}$
of $M$ by setting $M_d=R_d m_1 + \dots R_d m_k$ for $d\geq 0$ and $M_{-\infty}=\{0\}$. Having defined this we now obtain a \emph{gauge}
$\delta_M : M \rightarrow \mathbb{N} \cup {-\infty}$ defined as
$$\delta_M(m)=
\left\{
\begin{array}{ll}
-\infty, & \text{if } m=0 \\
d, & \text{if } m\neq 0 \text{ and } m\in M_d\setminus M_{d-1}
\end{array}
\right. .
$$
In particular any cyclic $R$-module has a natural gauge obtained by choosing as a set of generators the singleton
consisting of the image of $1$.

In this paper we shall call a quotient $S=R/I$ \emph{gauge bounded} if
there exist a set of homogeneous generators $\{\psi_\gamma\in \{\mathcal{C}_S\}_{e_\gamma}\}_{\gamma\in\Gamma}$
of $\mathcal{C}_S$ viewed as a right $S$-module
such that for some constant $K$ and all $r\in R$,
$\delta_S(\psi_\gamma(r+I)) \leq \delta_S(r+I) p^{-e_\gamma} + K$ (cf.~\cite[Definition 4.7]{BlickleTestIdealsViaAlgebras}.)

\begin{thm}[Corollary 4.16 in \cite{BlickleTestIdealsViaAlgebras}]\label{Theorem: gauge-boundedness implies discreteness}
If $S=R/I$ is \emph{gauge bounded}, the set of $F$-jumping coefficients of any ideal in $S$ is discrete.
\end{thm}

Recall our assumption that $R$ be $F$-finite; this amounts to the finiteness of the field extension $\mathbb{K}\subset \mathbb{K}^{1/p^e}$ for all (equivalently, some) $e\geq 0$.
For $e\geq 0$ fix a $\mathbb{K}$-basis $\mathcal{B}_e$ of $\mathbb{K}^{1/p^e}$ and assume further that $1\in \mathcal{B}_e$.

Recall (e.g., from section 1 of \cite{FedderFPureRat}) that
$F_*^e R\cong \mathbb{K}^{1/p^e}[x_1^{1/p^e}, \dots, x_n^{1/p^e}]$ and our assumption that $R$ is $F$-finite implies that
$F_*^e R$ is a free (left) $R$-module with free basis
$\{ b x_1^{\alpha_1/p^e} \dots x_n^{\alpha_n/p^e} \,|\, b\in \mathcal{B}_e, 0\leq \alpha_1, \dots, \alpha_n <p^e \}$.

We introduce the following notation: any $\mathbf{\alpha}=(\alpha_1, \dots \alpha_n)\in \mathbb{N}^n$ let $x^\mathbf{\alpha}$ denote the monomial
$x_1^{\alpha_1} \dots x_n^{\alpha_n}$ in $R$ and let $x^{\mathbf{\alpha}/p^e}$ denote the monomial $x_1^{\alpha_1/p^e} \dots x_n^{\alpha_n/p^e}$ in
$F_*^e R\cong \mathbb{K}^{1/p^e}[x_1^{1/p^e}, \dots, x_n^{1/p^e}]$. We shall also denote the condition
$a\leq \alpha_1, \dots \alpha_n < b$ with $a\leq \mathbf{\alpha}< b$ and the equalities
$\alpha_1= \dots =\alpha_n =a$ with $\mathbf{\alpha}=a$.
For any $e\geq 0$ let $\Lambda_e=\{ (\alpha_1, \dots, \alpha_n) \,|\, 0\leq \alpha_1, \dots, \alpha_n<p^e \}$.

We define the \emph{trace map} $T: F_*^e R \rightarrow R$ to be the projection onto the free summand $R x_1^{(p^e-1)/p^e} \dots  x_n^{(p^e-1)/p^e}$
and recall that there is an isomorphism $F_*^e R \rightarrow \Hom_R(F_*^e R, R)$
sending $r$ to the composition $T \circ \mu_{r^{1/p^e}}$ where $\mu_{r^{1/p^e}}: F_*^e R \rightarrow F_*^e R$ is given by multiplication by $r^{1/p^e}$ on the right.
For a quotient $S=R/I$ we have an explicit expression for $\mathcal{C}_S$:
the results in \cite[\S 1]{FedderFPureRat} imply that each $\Hom_S(F_*^e S, S)$ with its right $S$-module structure is isomorphic to
$\displaystyle\frac{(I^{[p^e]}:I)}{I^{[p^e]}} \overline{T}$ where $a \overline{T}(r+I)=T(a r)+I$ for all $a\in (I^{[p^e]}:I)$.

\begin{lem}\label{Lemma: gauge of generators and gauge-boundedness}
The hypothesis of Theorem \ref{Theorem: gauge-boundedness implies discreteness} will be satisfied
if we can find a constant $K$  and, for all $e\geq 0$,  a set of generators $g_1, \dots, g_{\nu_e}$ of $(I^{[p^e]}:I)$ such that
$\delta_R(g_i)=\delta_S(g_i + I)\leq K p^e$ for all $1 \leq i \leq g_{\nu_e}$.
\end{lem}
\begin{proof}
Fix any $e\geq 0$ and write $q=p^e$.
First note that for any $r\in R$, $ \delta_{R} (r)=\delta_{F_*^e R} (r)/q$.

Now pick $g=g_i$ for some $1\leq i\leq {\nu_e}$ and write
$$g^{1/q}=\sum_{\mathbf{\alpha}\in \Lambda_e, v\in \mathcal{B}_e} g_{\mathbf{\alpha},v} v x^{\mathbf{\alpha}/q}$$
where each $g_{\mathbf{\alpha},v}$ is in $R$.
For any $r^{1/q}\in \mathbb{K}^{1/p^e}[x_1^{1/p^e}, \dots, x_n^{1/p^e}]$ write
$$r^{1/q}=\sum_{\mathbf{\beta}\in \Lambda_e , w\in \mathcal{B}_e} r_{\mathbf{\beta},w} w x^{\mathbf{\beta}/q}$$
where each $r_{\mathbf{\beta},w}$ is in $R$ and compute
\begin{eqnarray*}
\delta_R \left( T\circ \mu_{g^{1/q}} (r^{1/q}) \right)& = &
 \delta_R \left( \sum_{ \mathbf{\alpha}, \mathbf{\beta}\in \Lambda_e, v,w \in \mathcal{B}_e, \atop{\mathbf{\alpha}+\mathbf{\beta}=q-1, vw\in \mathbb{K}}} g_{\mathbf{\alpha},v} r_{\mathbf{\beta},w} \right) \\
&\leq & \max \{  \delta_R\left(g_{\mathbf{\alpha},v} r_{\mathbf{\beta},w} \right) \,|\, \mathbf{\alpha}, \mathbf{\beta}\in \Lambda_e, v,w \in \mathcal{B}_e, \mathbf{\alpha}+\mathbf{\beta}=q-1, vw\in \mathbb{K}\}\\
\end{eqnarray*}
Now  each $\delta_R\left(g_{\mathbf{\alpha},v} r_{\mathbf{\beta},w} \right)$ above is at most
$\delta_R\left(g_{\mathbf{\alpha},v}\right) + \delta_R\left(r_{\mathbf{\beta},w} \right)$ (cf.~\cite[Lemma 4.1]{BlickleTestIdealsViaAlgebras},)
$$\delta_R\left(g_{\mathbf{\alpha},v}\right) =
\delta_{F_*^e R}\left(g_{\mathbf{\alpha},v}\right)/q \leq
\delta_{F_*^e R}\left(g_{\mathbf{\alpha},v} v x^{\mathbf{\alpha}/q} \right)/q \leq
\delta_{F_*^e R}\left( g^{1/q} \right)/q\leq K$$
and
$$\delta_R\left(r_{\mathbf{\alpha},w}\right) =
\delta_{F_*^e R}\left(r_{\mathbf{\alpha},w}\right)/q \leq
\delta_{F_*^e R}\left(r_{\mathbf{\alpha},w} w x^{\mathbf{\beta}/q} \right)/q \leq
\delta_{F_*^e R}\left( r^{1/q} \right)/q$$
hence
$$\delta_R \left( T\circ \mu_{g^{1/q}} (r^{1/q}) \right)\leq K+\delta_{F_*^e R}\left( r^{1/q} \right)/q .$$
\end{proof}

\begin{rmk}
\label{Remark: homogenization}
It is clear that, as the gauge of any element in $R$ is at most its degree, the condition in Lemma
\ref{Lemma: gauge of generators and gauge-boundedness} will be satisfied if there is a constant $K$ such that the maximal degree of all minimal
generators of $(I^{[p^e]}:I)$ is bounded by $K p^e$.
If $I\subseteq R$ is a homogeneous ideal, then the maximal degree of all minimal generators of $(I^{[p^e]}:I)$
is bounded by the Castelnuovo-Mumford regularity of $(I^{[p^e]}:I)$.

If $I\subseteq R$ is not a homogeneous ideal, we may consider the homogenization
$h(I)$ of $I$ with respect to a new variable.
If the maximal degree of all minimal generators of $((h(I))^{[p^e]}:h(I))$ is bounded by $Kp^e$,
then by de-homogenizing  one can see that
$(I^{[p^e]}:I)$ can be generated by polynomials whose degrees are bounded by $Kp^e$.
\end{rmk}

We shall assume henceforth that $I\subseteq R$ is a homogeneous ideal.
Following Remark \ref{Remark: homogenization}, in order to find a uniform $K$ (independent of $e$) such that the
maximal degree of a minimal generator of $(I^{[p^e]}:I)$ is bounded by $K p^e$, we shall attempt a harder problem, namely,
the bounding of the Castelnuovo-Mumford regularity of these homogenous ideals. But first review very briefly the notion of
Castelnuovo-Mumford regularity  and some of its properties.

Recall that the Castelnuovo-Mumford regularity of a finitely generated graded $R$-module $M$ (which we shall denote $\reg(M)$) is defined in terms of
its minimal graded resolutions $F_\bullet$.
To wit, for all $i\geq 0$ write $F_i=\oplus_{j=1}^{\rank(F_i)} R(-d_{i j})$ where $R(-d)$ denotes the degree shift $R(-d)_j=R_{d+j}$;
$\reg(M)$ is then defined as $\max \{ d_{i j}- i \,|\, i\geq 0, 1\leq j\leq \rank(F_i) \}$.
Note that the shifts $d_{0,j}$ are the degrees of a set of homogeneous minimal generators of  $M$, hence the maximal degree of
a homogenous minimal generator of $M$ is bounded by $\reg(M)$.

Alternatively, one may define Castelnuovo-Mumford regularity in terms of local cohomology.
If $M=\oplus M_d$ is a graded Artinian $R$-module, we have $\reg M=\max\{ d \,|\, M_d\neq 0\}$ and for a general
graded $R$-module we have
$$\reg M=\max_{i\geq 0} \{\reg \Tor_i(M,R/\mathfrak{m})-i\}=\max_{j\geq 0} \{\reg \HH_\mathfrak{m}^j(M)+j\}$$
(cf.~\cite[Corollary 4.5]{EisenbudGeometryOfSyzygies}).
Two immediate consequences of this characterization of regularity are the fact that if
$M_1 \rightarrow M_2 \rightarrow M_3$ is a graded exact sequence of Artinian $R$-modules, then $\reg M_2 \leq \max\{\reg M_1, \reg M_3\}$, and if
$0\rightarrow N_1 \rightarrow N_2 \rightarrow N_3 \rightarrow 0$ is a graded short exact sequence of $R$-modules, then
$\reg N_2 \leq \max\{\reg N_1, \reg N_3\}$
(cf.~\cite[Corollary 20.19]{EisenbudCommutativeAlgebra}.)

The rest of this paper will explore instances of graded ideals $I\subseteq R$
for which there exists a constant $K$ such that $\reg(I^{[p^e]} : I) \leq K p^e$ for all $e\geq 0$; these ideals will satisfy the hypothesis of Lemma \ref{Lemma: gauge of generators and gauge-boundedness}
and we will be able to deduce the discreteness of the sets of $F$-jumping coefficients of all ideals in $S=R/I$.
For the sake of readability we shall abbreviate the expression of the condition above as
$\reg(I^{[p^e]} : I) = \mathcal{O}(p^e)$. We can now claim the following.

\begin{cor}\label{Corollary: regularity and gauge-boundedness}
If $\reg(I^{[p^e]} : I) = \mathcal{O}(p^e)$ then $R=S/I$ is gauge-bounded. Hence the sets of jumping-coefficients of all ideals of $R$ are discrete.
\end{cor}

%%%%%%%%%%%%%%%%%%%%%%%%%%%%%%%%%%%%%%%%%%%%%%%%%%%%%%%%%%%%%%%%%%%%%%%%%%%%%%%%%%%%%%%%%%%%%%%%%%%%%%%%%%%%%%%%%
\section{The main results}
%%%%%%%%%%%%%%%%%%%%%%%%%%%%%%%%%%%%%%%%%%%%%%%%%%%%%%%%%%%%%%%%%%%%%%%%%%%%%%%%%%%%%%%%%%%%%%%%%%%%%%%%%%%%%%%%%

In this section we establish the conditions of Corollary \ref{Corollary: regularity and gauge-boundedness}
on the growth of regularity in some interesting cases.

Throughout this section we fix a minimal set of homogeneous generators
$\{ g_1, \dots, g_\nu\}$ for the homogenous ideal $I\subseteq R$.
We denote $d_i=\deg g_i$ for all $1\leq i\leq \nu$ and
we define the graded module
$$B=\frac{\oplus_{i=1}^{\nu} R(d_i)}
{R
\Gamma
}
$$
where
$$\Gamma= \left(\begin{array}{c} g_1\\ \vdots\\ g_\nu \end{array} \right )$$
together with its graded free resolution $F_\bullet$
$$ 0 \rightarrow R \xrightarrow[]{\Gamma} \oplus_{i=1}^{\nu} R(d_i) \rightarrow 0 .$$

For any $e\geq 0$ we may use this graded resolution to compute $T_j^e=\Tor_j^R(R/I^{[p^e]}, B)$ and we obtain
$T^e_1=(I^{[p^e]} : I)/ I^{[p^e]}$, $T_0^e=B\otimes R/I^{[p^e]}$ and $T^e_j=0$ for all $j>1$.
Since
$$\frac{R/I^{[p^e]}}{(I^{[p^e]} : I)/ I^{[p^e]}} \cong \frac{R}{(I^{[p^e]} : I)}$$
and
$$\reg \frac{R/I^{[p^e]}}{(I^{[p^e]} : I)/ I^{[p^e]}} = \reg  (I^{[p^e]} : I)/ I^{[p^e]} + 1$$
we deduce that
$\reg {(I^{[p^e]} : I)} = \reg \frac{R}{(I^{[p^e]} : I)} - 1 = \mathcal{O}(p^e)$
if and only if $\reg T_1^e = \mathcal{O}(p^e)$.

The one dimensional case is straightforward.

\begin{thm}
If $\dim R/I=1$,
all sets of $F$-jumping coefficients of ideals in $R/I$ are discrete.
\end{thm}
\begin{proof}
In view of Corollary \ref{Corollary: regularity and gauge-boundedness} and the discussion above,
we need to show that $\reg T_1^e = \mathcal{O}(p^e)$.

Consider the short exact sequences
$0 \rightarrow T_1^e \rightarrow R/I^{[p^e]} \rightarrow C \rightarrow 0$ where $C$ is the cokernel of the first non-zero map
and $0\rightarrow C \rightarrow \oplus_{i=1}^{\nu} R/I^{[p^e]}(d_i) \rightarrow T_0^e \rightarrow 0$.
These induce long exact sequences
\begin{equation}\label{eqn1}
0  \rightarrow \HH^0_\mathfrak{m} (T_1^e) \rightarrow \HH^0_\mathfrak{m} (R/I^{[p^e]}) \rightarrow \HH^0_\mathfrak{m}(C) \rightarrow \dots
\end{equation}
and
\begin{equation}\label{eqn2}
0 \rightarrow \HH^0_\mathfrak{m}(C) \rightarrow \oplus_{i=1}^{\nu} \HH^0_\mathfrak{m}( R/I^{[p^e]})(d_i) \rightarrow \HH^0_\mathfrak{m}(T_0^e) \rightarrow \dots .
\end{equation}

Note that since $R$ is regular, one may obtain a minimal graded free resolution of $R/I^{[p^e]}$ by applying the Frobenius functor to
a minimal graded free resolution of $R/I$, and hence $\reg R/I^{[p^e]}= \mathcal{O}(p^e)$.
The inclusions
$\HH^0_\mathfrak{m} (T_1^e) \hookrightarrow \HH^0_\mathfrak{m} (R/I^{[p^e]})$
and
$\HH^0_\mathfrak{m} (C) \hookrightarrow \HH^0_\mathfrak{m}(\oplus_{i=1}^{\nu} R/I^{[p^e]})(d_i)$
now
imply
$\reg \HH^0_\mathfrak{m} (T_1^e)=\mathcal{O}(p^e)$ and
$\reg \HH^0_\mathfrak{m} (C)=\mathcal{O}(p^e)$.
The last equality in turn implies that $\reg \HH^1_\mathfrak{m} (T_1^e)=\mathcal{O}(p^e)$.
If $\dim R/I=1$, these calculations show that $\reg T_1^e = \mathcal{O}(p^e)$ and the result follows.
\end{proof}

We now tackle the general case. We fix homogeneous elements $u_1, \dots, u_N\in R$  of degrees $d_1, \dots , d_N$.
For $1\leq k\leq N$ let
$$U_k=\left( \begin{array}{c} u_1\\ \vdots\\ u_k\end{array}\right) \in R^k .$$
For all $1\leq k\leq N$ we can define the graded module
$B_k=\frac{\oplus_{i=1}^{\nu} R(d_i)}
{R U_k}
$
together with its graded free resolution
%%% MK $F^{(k)}_\bullet$
\[ 0 \rightarrow R \xrightarrow[]{U_k} \oplus_{i=1}^{k} R(d_i) \rightarrow 0 .\]

For any $e\geq 0$ we may use this graded resolution to compute $T_{k j}^e=\Tor_j^R(R/I^{[p^e]}, B_k)$.
Just as in the previous theorem we have two short exact sequences
$0 \rightarrow T_{k 1}^e \rightarrow R/I^{[p^e]} \rightarrow C_k \rightarrow 0$
and $0\rightarrow C_k \rightarrow \oplus_{i=1}^{k} R/I^{[p^e]}(d_i) \rightarrow T_{k 0}^e \rightarrow 0$
which induce long exact sequences
\begin{equation}\label{eqn3}
\dots  \rightarrow \HH^j_\mathfrak{m} (T_{k 1}^e) \rightarrow \HH^j_\mathfrak{m} (R/I^{[p^e]}) \rightarrow \HH^j_\mathfrak{m}(C_k) \rightarrow \dots
\end{equation}
and
\begin{equation}\label{eqn4}
\dots \rightarrow \HH^j_\mathfrak{m}(C_k) \rightarrow \oplus_{i=1}^{k} \HH^j_\mathfrak{m}( R/I^{[p^e]})(d_i) \rightarrow \HH^j_\mathfrak{m}(T_{k 0}^e) \rightarrow \dots .
\end{equation}
These imply that $\reg \HH^0_\mathfrak{m}(C_k) = \mathcal{O}(p^e)$ and $\reg \HH^1_\mathfrak{m}(T_{k 1}^e) = \mathcal{O}(p^e)$.
Set $J_k$ to be the ideal in $R$ generated by $u_1,\dots,u_k$ for all $1\leq k\leq N$;
we have $T_{k 1}^e=(\Ipe:J_k)/\Ipe$. Hence from the short exact sequence $0\to (\Ipe:J_k)/\Ipe\to R/\Ipe\to R/(\Ipe:J_k)\to 0$ we can deduce that
\begin{equation}\label{eqn5}
\reg \HH^1_\mathfrak{m}\left( R/(\Ipe: J_k) \right)=\mathcal{O}(p^e).
\end{equation}

\begin{lem}
\label{Lemma: computation}
Assume  that
$\reg (R/(\Ipe + u_i R)) =\mathcal{O}(p^e)$ for all $1\leq i\leq N$.
Then for all $1\leq k\leq N$ and all $j\geq 0$
$$\reg \HH^j_\mathfrak{m} \left(\frac{\oplus_{i=1}^k (R/\Ipe)(d_i)}{R/\Ipe U_k} \right)  = \mathcal{O}(p^e) .$$
\end{lem}
\begin{proof}
We proceed by induction on $j$. First we prove the case when $j=0$ which we will do by induction on $1\leq k\leq N$, the case $k=1$ being immediate from the hypothesis.
Assume that $k>1$ and that the lemma holds for all smaller values of $k$.
Fix $e\geq 0$ and abbreviate $A=R/\Ipe$.
We apply the induction hypothesis to the graded short exact sequence
\begin{equation}
\label{eqn6}
0 \rightarrow
\frac{A U_{k-1} \oplus A u_k}{A U_k} \xrightarrow[]{}
\frac{\oplus_{i=1}^k A(d_i)}{A U_k} \rightarrow
\left(\frac{\oplus_{i=1}^{k-1} A(d_i)}{A U_{k-1}} \right) \oplus \frac{A}{A u_k }(d_k) \rightarrow 0
\end{equation}
and reduce the problem to the bounding of the regularity of
$$\HH^0_\mathfrak{m}\left( \frac{A U_{k-1} \oplus A u_k}{A U_k} \right) .$$

For all $1\leq j\leq k$ let $J_j=u_1R+\dots + u_jR$.
We have a graded map
$$\phi:  A \rightarrow \frac{A U_{k-1} \oplus A u_k}{A U_k} $$
which sends $a$ to the image of $a U_{k-1} \oplus 0$ in $\frac{A U_{k-1} \oplus A u_k}{A U_k}$. Note that in $\frac{A U_{k-1} \oplus A u_k}{A U_k}$ we always have
\[(u_1,\dots,u_{k-1},0)+(0,\dots,0,u_k)=1\cdot (u_1,\dots,u_k)=1\cdot U_k=0.\]
Hence, for any given element $(au_1,\dots,au_{k-1},a'u_k)\in \frac{A U_{k-1} \oplus A u_k}{A U_k}$, we have
\begin{align}
(au_1,\dots,au_{k-1},a'u_k)&=a\cdot (u_1,\cdots,u_{k-1},0)+a'\cdot (0,\dots,0,u_k)\notag\\
&=a\cdot (u_1,\cdots,u_{k-1},0)+(-a')\cdot (u_1,\cdots,u_{k-1},0)\notag\\
&=(a-a')\cdot (u_1,\cdots,u_{k-1},0)=\phi(a-a')\notag
\end{align}
and so $\phi$ is surjective. Also note that
\begin{eqnarray*}
\ker \phi& = & \left\{ a \,|\, \left(\begin{array}{l} a u_1\\ \vdots \\ a u_{k-1} \\ 0  \end{array} \right) =
z \left(\begin{array}{l} u_1\\ \vdots \\ u_{k-1} \\ u_k  \end{array} \right) \text{ for some } z\in A \right\} \\
& = & \left((\Ipe: u_1) A \cap \dots \cap (\Ipe: u_{k-1})A\right) + (\Ipe: u_k) A\\
& = & \left(\Ipe: J_{k-1})\right)A + (\Ipe: u_k) A\\
\end{eqnarray*}
%{\color{blue} This is only to convince myself about the the kernel. If $a$ is in $\ker\phi$, then we must have $(a-z)u_i\in I^{[p^e]}$ for $1\leq %i\leq k-1$ and $zu_k\in \Ipe$. Hence $a\in  \left(\Ipe: J_{k-1}\right)A + (\Ipe: u_k) A$. On the other hand, if $a=r_1+r_2$ with $r_1\in  \left(\Ipe: %J_{k-1}\right)A$ and $r_2\in (\Ipe: u_k) A$, then we have $r_1u_i=0$ in $A$ and $r_2u_k=0$ in $A$. So $aU_{k-1}\oplus 0=r_2U_k$, which implies that %$a\in \ker\phi$.}
Using the graded isomorphism $\displaystyle A/\ker \phi \cong \frac{A U_{k-1} \oplus A u_k}{A U_k}$ we reduce the problem to the bounding of the regularity of
$$\HH^0_\mathfrak{m} \left( \frac{A}{(\Ipe: J_{k-1}R)A +  (\Ipe: u_{k}R) A } \right) .$$

%%%The short exact sequence
%%%\[0\to  \left(\Ipe: J_{k-1}\right)A + \left(\Ipe: u_k\right) A \to A \xrightarrow{\phi} \frac{A U_{k-1} \oplus A u_k}{A U_k} \to 0\]
%%%reduces the problem to the bounding of the regularity of
%%%$$\displaystyle\HH^1_\mathfrak{m} \left((\Ipe: J_{k-1})A + (\Ipe: u_k) A\right) .$$
%%%Then the exact sequence
%%%\[0\to  \left(\Ipe: J_{k-1}\right)A + (\Ipe: u_k) A \to A \to \frac{A}{(\Ipe: J_{k-1})A +  (\Ipe: u_{k}R) A } \to 0\]
%%%in turn reduces the problem to bounding the regularity of
%%%$$\HH^0_\mathfrak{m} \left( \frac{A}{(\Ipe: J_{k-1}R)A +  (\Ipe: u_{k}R) A } \right) .$$

Now the short exact sequence
$$
0 \rightarrow \frac{A}{(\Ipe: J_k)A} \rightarrow
\frac{A}{(\Ipe: J_{k-1}) A} \oplus \frac{A}{(\Ipe: u_{k}R) A}  \rightarrow
\frac{A}{(\Ipe: J_{k-1})A +  (\Ipe: u_{k}R) A } \rightarrow 0
$$
yields the exact sequence
\begin{eqnarray}
\label{eqn7}
&&\HH^i_\mathfrak{m} \left( \frac{A}{(\Ipe: J_{k-1}) A} \right) \oplus \HH^i_\mathfrak{m} \left( \frac{A}{(\Ipe: u_{k}R) A}  \right) \rightarrow \\
&&\HH^i_\mathfrak{m} \left( \frac{A}{(\Ipe: J_{k-1}R)A +  (\Ipe: u_{k}R) A } \right) \rightarrow \notag\\
&&\HH^{i+1}_\mathfrak{m} \left(  \frac{A}{(\Ipe: J_{k}R)A}  \right)\notag
\end{eqnarray}
which (with $i=0$) establishes
$\reg \HH^0_\mathfrak{m} \left( A/\left(\Ipe: J_{k-1}R\right)A +  (\Ipe: u_{k}R) A )\right) = \mathcal{O}(p^e)$ because of our induction hypothesis and (\ref{eqn5}).
This concludes the proof of the initial step of our induction, namely,  that of $j=0$.

Assume that we have proved that
\[\reg \HH^l_\mathfrak{m} \left(\frac{\oplus_{i=1}^k (R/\Ipe)(d_i)}{R/\Ipe U_k} \right)  = \mathcal{O}(p^e)\]
for $0\leq l\leq j-1$ and for all $1\leq k\leq N$. From the exact sequence (\ref{eqn4}), we can see that
\[\reg \HH^\gamma_\mathfrak{m}(C_k)= \mathcal{O}(p^e)\]
for all $0\leq \gamma\leq j$ and for all $1\leq k\leq N$. Then from the exact sequence (\ref{eqn3}), we can see that
\[\reg \HH^\delta_\mathfrak{m} (T_{k 1}^e)= \mathcal{O}(p^e) \]
for all $0\leq \delta\leq j+1$ and for all $1\leq k\leq N$, i.e.
\[\reg \HH^\delta_\mathfrak{m} (\frac{R}{(I^{[p^e]}:J_k)})= \mathcal{O}(p^e)\]
for all $0\leq \delta\leq j+1$.
%??????????, for all ideals $J_k$ that is generated by $k$ elements of $\{u_1,\dots,u_N\}$, and for all $1\leq k\leq N$.

Now we will prove that $ \HH^j_\mathfrak{m} \left(\frac{\oplus_{i=1}^k (R/\Ipe)(d_i)}{R/\Ipe U_k} \right) =\mathcal{O}(p^e)$
using induction on $k$. When $k=1$, this is precisely our assumption that $\reg \left(R/(\Ipe + u_i R)\right) =\mathcal{O}(p^e)$ for all $1\leq i\leq N$. Now the induction hypothesis applied to the exact sequence (\ref{eqn6}) reduces our proof to bounding the regularity of
\[\HH^i_\mathfrak{m}\left( \frac{A U_{k-1} \oplus A u_k}{A U_k} \right),\]
for all $0\leq i\leq j$. The exact sequence (\ref{eqn7}) finishes the proof.
\end{proof}

\begin{thm}
\label{Theorem: discreteness under linear regularity}
If $\reg\left(R/(\Ipe+g_i R)\right)=\mathcal{O}(p^e)$ for all $1\leq i\leq \nu$, then
all sets of $F$-jumping coefficients of all ideals in $R/I$ are discrete.
\end{thm}
\begin{proof}
Apply Lemma \ref{Lemma: computation} with $N=\nu$ and $u_k=g_k$ for all $1\leq k\leq \nu$
and deduce that $\reg T_{1}^e = \mathcal{O}(p^e)$.
\end{proof}

\begin{rmk}
The non-graded case may be obtained by combining Remark \ref{Remark: homogenization} and Theorem \ref{Theorem: discreteness under linear regularity}.
For an arbitrary radical ideal $J=(g_1,\dots,g_{\nu})$ in $R$ consider the homogenizations   $h(I)$ and $h(g_i)$ with respect to
a new variable $Z$.
If $\reg\left(R[Z]/(h(J)^{[p^e]}+h(g_i) R[Z])\right)=\mathcal{O}(p^e)$ for all
$1\leq i\leq \nu$, then $R/J$ would be gauge bounded and thus all sets of $F$-jumping coefficients of all ideals in $R/J$ would be discrete.
\end{rmk}

\begin{lem}[cf.~section 6 in \cite{KatzmanComplexity} and Theorem 4.2 in \cite{ChardinOnTheBehaviorOfCastelnuovo-MumfordRegularity}]
\label{Lemma: regularity of Frobenius}
Let $T$ be a standard graded algebra of prime characteristic $p$ and let $M$ be a finitely generated $T$-module.
Let $F_T(-)=(-) \otimes_T F^e_* T$ denote the Frobenius functor.
Let $\Sing T$ denote the singular locus of $T$.
If $\dim \left(\Sing T \cap \Supp M \right )\leq 1$ then $\reg F^e_T (M) = \mathcal{O}(p^e)$.
\end{lem}
\begin{proof}
Let $P\subset T$ be a prime of dimension at least two; we have $M_P=0$ or $T_P$ is regular.
If the latter holds, $F^e_* T$ is a flat left $T$-module, and in either case $\Tor_1^T(F^e_* T, M)_P=\Tor_1^T((F^e_* T)_P, M_P)=0$.
Thus $\dim \Tor_1^T(F^e_* T, M) \leq 1$ and the result follows from \cite[Theorem 2.3]{EisenbudHunekeUlrichTheRegularityOfTor}.
\end{proof}

\begin{cor}
\label{discreteness singular locus}
If $\dim \left(\Sing (R/g_iR) \cap V(I) \right)\leq 1$ for all $1\leq i\leq \nu$ then $\reg(R/(I^{[p^e]}+g_iR))= \mathcal{O}(p^e)$ for all $1\leq i\leq \nu$. Consequently, the sets of $F$-jumping coefficients of all ideals (homogeneous or not) in $S=R/I$ are discrete.
\end{cor}
\begin{proof}
Applying Lemma \ref{Lemma: regularity of Frobenius} to $T=R/g_i R$ and $M=R/I$ for all $1\leq i\leq \nu$, we have that
$\reg R/(\Ipe + u_iR)  =\mathcal{O}(p^e)$. Now Theorem \ref{Theorem: discreteness under linear regularity} finishes the proof.
\end{proof}

\begin{rmk}\label{Remark: two-dimensional domains}
The assumption that $\dim \left(\Sing (R/g_i R) \cap V(I) \right)\leq 1$ for all $1\leq i\leq \nu$ does not seem to be very restrictive.
For instance, when $\dim(R/I)=2$, if $I$ happens to be a prime ideal and $p$ does not divide the degree of any one of the generators $\{g_1,\dots,g_{\nu}\}$, then $\dim \left(\Sing (R/(g_i)) \cap V(I) \right)\leq 1$ is automatically satisfied.
\end{rmk}

%%%%%%%%%%%%%%%%%%%%%%%%%%%%%%%%%%%%%%%%%%%%%%%%%%%%%%%%%%%%%%%%%%%%%%%%%%%%%%%%%%%%%%%%%%%%%%%%%%%%%%%%
\section{An application}
%%%%%%%%%%%%%%%%%%%%%%%%%%%%%%%%%%%%%%%%%%%%%%%%%%%%%%%%%%%%%%%%%%%%%%%%%%%%%%%%%%%%%%%%%%%%%%%%%%%%%%%%

Throughout this section we fix the following notation
$R=\mathbb{K}[x,y,z,u,v,w]$,
$g_1=yu-xv$,
$g_2=zu-xw$,
$g_3=zv-yw$, and
$S=R/g_3 R$.
In this section we apply Theorem \ref{Theorem: discreteness under linear regularity} to the determinantal ring $R/(g_1R+g_2R+g_3R)$.
Our main aim is to illustrate the usefulness of this theorem with a relatively simple example;
however, the fact that this determinantal ring is gauge bounded is also shown in \cite{KSS} using direct methods.
Our method is perhaps less computational and might be more generalizable.
% MK
Although the argument below might seem to involve formidable computational insight, it is really the formalization of results inspired by
computations with Macaulay2 \cite{M2}. The main technical result in this section is a description of resolutions over an hypersurface (which
are eventually periodic of period two (cf.~\cite{EisenbudHomologicalAlgebra},)) one such for each value of $e\geq 0$.
In many cases, including other determinantal rings,
the computation of these resolutions for small values of $e$ reveals a pattern which can be formally described using methods similar
to the those used below.

\bigskip
Back to our example, we fix the reverse lexicographical order on the monomials of $R$ with $x>y>z>u>v>w$ and note that the terms
of $g_1$, $g_2$ and $g_3$ were listed in descending order.

We now also fix $q=p^e$ for some $e$; the bulk of this section will be devoted to the bounding of $\reg R/(g_1 R + g_2^q R + g_3^q R)$ and we do so
by computing an explicit free resolution of the $S$-module $S/(g_2^q S + g_3^q S)$.

For all $0\leq j\leq q$ define $h_j=x^j z^q u^{q-j} v^j - x^q y^j w^q$; note that $h_0=g_2^q$ and $h_q=x^q g_3^q$.

\begin{lem}\hfil
\begin{enumerate}
\item[(a)]
For any $a,b\in R$, let $\mathcal{S}(a,b)$ denote the S-polynomial of $a$ and $b$.
\begin{enumerate}
\item[(i)]
For all $0\leq i\leq j \leq q$ we have
\begin{eqnarray*}
\mathcal{S}(h_j, h_i)=
u^{j-i} h_j - x^{j-i} v^{j-i} h_i & = &
x^q w^q  (x^{j-i} y^i v^{j-i} - y^j u^{j-i})\\
& = &
-x^q w^q \sum_{k=1}^{j-i} x^{k-1} y^{j-k} u^{j-i-k} v^{k-1} g_1 .
\end{eqnarray*}
\item[(ii)]
For all $0\leq j\leq q$ we have
\begin{eqnarray*}
\mathcal{S}(h_j, g_3^q)=
v^{q-j}h_j - x^ju^{q-j} g_3^q & = &
w^q( x^j y^q u^{q-j} - x^q y^j v^{q-j})\\
&=& w^q \sum_{k=1}^{q-j} x^{j-1+k} y^{q-k} u^{q-j-k} v^{k-1}g_1 .
\end{eqnarray*}
\item[(iii)]
For all $0\leq j\leq  q-2$ we have $\mathcal{S}(h_j, g_1)=y h_j - x^j z^q u^{q-j-1} v^j g_1=h_{j+1}$.
\item[(iv)]
$\mathcal{S}(h_{q-1}, g_1)=y h_{q-1} - x^{q-1} z^q  v^{q-1} g_1= x^q g_3^q$.
\item[(v)]
$\mathcal{S}(g_3^q, g_1)=yu g_3^q- z^q v^q g_1$.
\end{enumerate}

\item[(b)] Let $G$ be the row vector $[ h_0, \dots, h_{q-1}, g_3^q, g_1]$.
Then the entries in $G$ form a Gr\"obner basis for the ideal $g_1 R+ g_2^q R+ g_3^q R$.
\item[(c)] $G=[g_2^q, g_3^q, g_1] T$ where $T$ is a $3\times (q+2)$ matrix of the form
$$
\left[
\begin{array}{ccccccc}
1 & y & y^2 &       & y^{q-1} & 0 & 0\\
0 & 0 & 0   & \dots & 0       & 1 & 0\\
0 & * & *   &       & *       & 0 & 1\\
\end{array}
\right]
.$$
\item[(d)]
For all $1\leq \ell \leq q+2$,
let $\mathbf{e}_\ell$ denote the $\ell^\text{th}$-standard vector in $R^{q+2}$.
The module of syzygies of $G$ is generated by
vectors of the form
\begin{enumerate}
\item[(i)]
$V_{i, j}=x^{j-i} v^{j-i} \mathbf{e}_i - u^{j-i} \mathbf{e}_j + (\star) \mathbf{e}_{q+2}$ for all $1\leq i<j \leq q$,
\item[(ii)]
$V_{i, q+1}=v^{q-i+1} \mathbf{e}_i - x^i u^{q-i} \mathbf{e}_{q+1} + (\star) \mathbf{e}_{q+2}$ for all $1\leq i<  q+1$,
\item[(iii)]
$V_{i, q+2}=y \mathbf{e}_i -   \mathbf{e}_{i+1} + (\star) \mathbf{e}_{q+2}$ for all $1\leq i<  q$ ,
\item[(iv)]
$V_{q, q+2}=y \mathbf{e}_q   -x^q\mathbf{e}_{q+1}  + (\star) \mathbf{e}_{q+2}$,
\item[(v)]
$V_{q+1, q+2}=g_1 \mathbf{e}_{q+1} -  g_3^q  \mathbf{e}_{q+2}$

\end{enumerate}
where $(\star)$ denote unspecified elements in $R$.

\end{enumerate}
\end{lem}
\begin{proof}
Statements (a)(i) and (a)(ii) follow from the telescoping sums
$$ \sum_{k=1}^{j-i} x^{k-1} y^{j-k} u^{j-i-k} v^{k-1} (yu-xv) = y^j u^{j-i} - x^{q+j-i} y^i v^{j-i} $$
and
$$ \sum_{k=1}^{q-j} x^{k+j-1}  y^{q-k}  u^{q-k-j}  v^{k-1}  (yu-xv)=
x^j y^q u^{q-j} - x^q y^j v^{q-j},$$
respectively.
The rest of the statements in (a) follow from straightforward calculations.

We now know that the S-polynomials between the elements of $G$ reduce to zero with respect to the elements of $G$ and hence (b)
follows from Buchberger's Theorem (cf.~\cite[\S 1.7]{AdamsandLoustaunau}).

Part (c) follows inductively from the fact that $h_{j+1} =\mathcal{S}(h_j, g_1)\equiv  y h_j$ modulo $g_1$.

Part (d) follows  \cite[Theorem 3.4.1]{AdamsandLoustaunau} and the corresponding calculation in part (a).
\end{proof}

\begin{lem}\label{Lemma: syzygies 1}
The module of syzygies of the image of $(g_2^q, g_3^q)$ in $S$ is generated by the vectors
$$
\left[
\begin{array}{c}
y^{j} v^{q-j}\\ x^j u^{q-j}
\end{array}
\right]
$$
with $0\leq j\leq q$.
\end{lem}
\begin{proof}
We apply  \cite[Theorem 3.4.3]{AdamsandLoustaunau} and conclude that
the module of syzygies of $(g_2^q, g_3^q, g_1)$ is generated by
$T V_{i,j}$ for all $1\leq i< j \leq q+2$, with the vectors $V_{i,j}$ as in the previous lemma.

\begin{enumerate}
\item[(i)] For all $1\leq i<j \leq q$
$$T V_{i, j}=\left[
\begin{array}{c}
y^{i-1} x^{j-i} v^{j-i} - y^{j-1} u^{j-i} \\
0\\
*
\end{array}
\right] \equiv
\left[
\begin{array}{c}
y^{i-1} y^{j-i} u^{j-i} - y^{j-1} u^{j-i} \\
0\\
*
\end{array}
\right] =
\left[
\begin{array}{c}
0 \\
0\\
*
\end{array}
\right]  \pmod{g_1}
.$$

\item[(ii)] For all $1\leq i<  q+1$,
$$T V_{i, q+1}
=\left[
\begin{array}{c}
y^{i-1} v^{q-i+1}\\
-x^i u^{q-i} \\
*\\
\end{array}
\right]
.$$

\item[(iii)]
For all $1\leq i<  q$
$$T V_{i, q+2}=
\left[
\begin{array}{c}
0\\
0 \\
*\\
\end{array}
\right]
.$$

\item[(iv)]
$$TV_{q, q+2}=
\left[
\begin{array}{c}
y^{q} \\
-x^q  \\
*\\
\end{array}
\right]
.$$

\item[(v)]
$$T V_{q+1, q+2}=
\left[
\begin{array}{c}
0\\
g_1 \\
*\\
\end{array}
\right]
\equiv
\left[
\begin{array}{c}
0\\
0 \\
*\\
\end{array}
\right]
\pmod{g_1}
.$$

\end{enumerate}

\end{proof}

\begin{lem}\label{Lemma: syzygies 2}
Let
$$M=
\left[
\begin{array}{cccccc}
v^q & y v^{q-1} & \dots & y^j v^{q-j} &\dots &y^q  \\
-u^q & -x u^{q-1} & \dots & -x^j u^{q-j} & \dots & -x^q  \\
\end{array}
\right]
.$$
The module of syzygies of the columns of $M$ viewed as vectors in $S^2$ is generated by
the vectors
$W_i=x \mathbf{e}_i - u \mathbf{e}_{i+1}$ and
$U_i=y \mathbf{e}_i - v \mathbf{e}_{i+1}$
for all $1\leq i\leq q$ where $\mathbf{e}_\ell$ denote the elementary vectors in $R^{q+1}$.
\end{lem}

\begin{proof}
We extend our monomial order from $R$ to $R^2$  by using its term-over-position extension (cf.~\cite[\S 3.5]{AdamsandLoustaunau}).
For all $0\leq i < j\leq q$ we compute the following $S$-polynomials
\begin{eqnarray*}
\mathcal{S}\left(
\twoVector{y^j v^{q-j}}{ - x^j u^{q-j}}, \twoVector{y^i v^{q-i}}{ - x^i u^{q-i} } \right)&=&
u^{j-i} \left( \twoVector{y^j v^{q-j}}{- x^j u^{q-j} }\right) - x^{j-i} \left(\twoVector{y^i v^{q-i}}{- x^i u^{q-i} }\right)\\
& = &  \sum_{k=1}^{j-1} x^{k-1} y^{j-k} u^{j-i-k} v^{q-j+k-1} \twoVector{  g_1}{0}
\end{eqnarray*}
where the last equality follows from the telescopic series
$$\sum_{k=1}^{j-i} x^{k-1} y^{j-k} u^{j-i-k} v^{q-j+k-1} g_1= y^j u^{j-i} v^{q-j} - x^{j-i} y^i v^{q-i} .$$
For all $0\leq i \leq q$ we have
$\mathcal{S}\left(\twoVector{y^i v^{q-i}}{ - x^i u^{q-i} }, \twoVector{ g_1 }{0}\right)=0$.
For all $0\leq i < q$
\begin{eqnarray*}
\mathcal{S}\left( \twoVector{y^i v^{q-i}}{- x^i u^{q-i}},  \twoVector{0}{g_1}\right) & = &
y \twoVector{y^i v^{q-i} }{ - x^i u^{q-i}} + x^i u^{q-i-1} \twoVector{0}{g_1 }\\
& = & v \twoVector{y^{i+1} v^{q-i-1}}{ - x^{i+1} u^{q-i-1} }
\end{eqnarray*}
and
$$\mathcal{S}\left(
\twoVector{y^q}{ - x^q },  \twoVector{0}{g_1}
\right)=
yu\twoVector{y^q}{ - x^q  } + x^q\twoVector{0}{g_1}=
xv \twoVector{y^q }{ - x^q} + y^q \twoVector{g_1}{0}.$$

We can now conclude that the columns of $M$ viewed as vectors in $R^2$ together with the vectors
$ \twoVector{g_1}{0}$ and  $\twoVector{0}{g_1}$ form a Gr\"obner bases.
We may now apply (\cite[Theorem 3.7.3]{AdamsandLoustaunau}) and list the syzygies among the columns in $M$ viewed as elements in $S^2$ as
$u^{j-i} \mathbf{e}_j - x^{j-i} \mathbf{e}_i$
for all $1\leq i<j\leq q$,
and
$ y \mathbf{e}_i - v \mathbf{e}_{i+1} $.

We now note that for all $1\leq i<j\leq q$,
$$u^{j-i} \mathbf{e}_j - x^{j-i} \mathbf{e}_i=- \sum_{k=1}^{j-i} x^{k-1} u^{j-i-k} W_j $$
and the result follows.
\end{proof}

\begin{thm}
The $S$-module $S/(g_2^qS+g_3^qS)$
has a minimal free resolution
$$
\begin{array}{c}
\dots \xrightarrow[]{\phi_5} S^{2q}(-3q-2k) \xrightarrow[]{\phi_4} S^{2q}(-3q-2k+1) \xrightarrow[]{\phi_5}  \dots\\
\xrightarrow[]{\phi_5} S^{2q}(-3q-2) \xrightarrow[]{\phi_4} S^{2q}(-3q-1)  \xrightarrow[]{\phi_3}\\
S^{q+1}(-3q)  \xrightarrow[]{\phi_2}
S^{2}(-2q)  \xrightarrow[]{\phi_1}
S^{1}  \xrightarrow[]{} 0
\end{array} .
$$
\end{thm}
\begin{proof}
We take $\phi_1=[g_2^q, g_3^q]$; the matrix $\phi_2$ is then the one whose columns are described in the statement on Lemma \ref{Lemma: syzygies 1}
and the matrix $\phi_3$ is the one whose columns consist of the first $q+1$ coordinates of the vectors $U_i, W_i$ ($1\leq i\leq q$)
described in the statement on Lemma \ref{Lemma: syzygies 2}.

The syzygies on the columns of $\phi_3$ arise from the relations
$x U_i - y W_i\equiv 0 \pmod{g_1}$ and
$u U_i - v W_i\equiv 0 \pmod{g_1}$ for all $1\leq i \leq q$ and
so the columns of $\phi_4$ are given by
$x \mathbf{e}_{2j-1} - y \mathbf{e}_{2j}$ and
$u \mathbf{e}_{2j-1} - v \mathbf{e}_{2j}$ for all $1\leq j\leq q$
where $\mathbf{e}_\ell$ denotes the $\ell^\text{th}$ elementary vector in $S^{2q}$.

The syzygies on the columns of $\phi_4$ arise from the relations
$u\left(x \mathbf{e}_{2j-1} - y \mathbf{e}_{2j}\right) - x \left(u \mathbf{e}_{2j-1} - v \mathbf{e}_{2j}\right)$ and
$v\left(x \mathbf{e}_{2j-1} - y \mathbf{e}_{2j}\right) -y \left(u \mathbf{e}_{2j-1} - v \mathbf{e}_{2j}\right)$
for all $1\leq j\leq q$
so the columns of $\phi_5$ are given by
$u \mathbf{e}_{2j-1} - x \mathbf{e}_{2j}$ and
$v \mathbf{e}_{2j-1} - y \mathbf{e}_{2j}$ for all $1\leq j\leq q$.

The syzygies on the columns of $\phi_5$ is now given by the columns of $\phi_4$ hence at this point we get a period-2 linear resolution.
\end{proof}

\begin{rmk}[Base Change for Tor]
If $A\to B$ is a ring homomorphism, then there is a homology spectral sequence (\cite[Theorem 5.6.6]{Wei94})
\[E^2_{i,j}=\Tor^B_i(\Tor^A_j(M,B),N)\Rightarrow \Tor^A_{i+j}(M,N),\]
for each $A$-module $M$ and each $B$-module $N$.

When $B=A/fA$ where $f$ is a non-zero-divisor in $A$, it is clear that $\Tor^A_i(M,B)=0$ except for $i=0,1$. Also,
$\Tor^A_0(M,B)=M/fM$ and  $\Tor^A_1(M,B)=\Ann_M f$ and
the spectral sequence gives rise to a long exact sequence
\begin{equation}
\label{eqn8}
\cdots\to \Tor^B_{i-1}(\Ann_M f, N)\to \Tor^A_i(M,N) \to \Tor^B_i(M/fM,N)\to \Tor^B_{i-2}(\Ann_M f,N)\to\cdots
\end{equation}
\end{rmk}

\begin{cor}\label{Corollary: determinantal ideal}
All sets of jumping coefficients of ideals in $R/g_1R+g_2R+g_3R$ are discrete.
\end{cor}
\begin{proof}
Setting $A=R$, $B=S=R/g_1R$, $M=S/(g_2^qS+g_3^qS)=R/g_1R+g_2^qR+g_3^qR$, and $N=\mathbb{K}$ in the long exact sequence (\ref{eqn8}) we obtain
$$ \dots \rightarrow \Tor^S_{i-1}(S/(g_2^qS+g_3^qS), \mathbb{K}) \rightarrow \Tor^R_{i}(R/(g_1R+g_2^qR+g_3^qR), \mathbb{K}) \rightarrow
 \Tor^S_{i}(S/g_2^qS+g_3^qS, \mathbb{K})  \rightarrow \dots $$
where $\mathbb{K}$ denotes here the quotient of $S$ by its irrelevant ideal.
The previous theorem now implies that the degrees of elements in a graded $\mathbb{K}$-bases for $\Tor^R_{i}(R/(g_1R+g_2^qR+g_3^qR), \mathbb{K})$
are bounded by $3q$ hence $\reg  \Tor^R_{i}(R/(g_1R+g_2^qR+g_3^qR), \mathbb{K})=\mathcal{O}(q)$ and by symmetry we also have
$$\reg  \Tor^R_{i}(R/(g_1^qR+g_2R+g_3^qR), \mathbb{K}) =\reg  \Tor^R_{i}(R/(g_1^qR+g_2^qR+g_3R), \mathbb{K}) =\mathcal{O}(q) .$$
The corollary now follows from Theorem \ref{Theorem: discreteness under linear regularity}.
\end{proof}

\section*{Acknowledgments}

We thank Karl Schwede for his useful comments on an early version of this manuscript.

\end{document}